\documentclass[reqno]{amsart}

\usepackage{amsmath}
\usepackage{amsfonts}
\usepackage{amsthm}
\usepackage{amssymb}
\usepackage{graphicx}
\usepackage{graphics}
\usepackage{bm}
\usepackage{dsfont}
\usepackage{color}
\usepackage[font=footnotesize]{caption}
\numberwithin{equation}{section}

\newtheoremstyle{personal}%
{12pt}
{12pt}
{\slshape}
{}
{\bfseries}
{.}
{.5em}
{}
\theoremstyle{personal}%
\newtheorem{thm}{Theorem}[section]
\newtheorem{claim}{Claim}

\newtheorem{lem}[thm]{Lemma}
\newtheorem{prop}[thm]{Proposition}
\theoremstyle{definition}
\newtheorem{defn}{Definition}[section]

\newcommand{\N}{\mathds{N}}
\newcommand{\Z}{\mathds{Z}}
\newcommand{\R}{\mathds{R}}

\newcommand{\F}{\mathds{F}}
\newcommand{\Q}{\mathds{Q}}
\newcommand{\diff}{\mathrm{d}}
\newcommand{\dist}{\mathrm{dist}}
\newcommand{\Tan}{\mathrm{T}}

\DeclareMathOperator{\grad}{\mathrm{grad}}

\newcommand{\Hom}{\mathrm{H}}
\newcommand{\Loc}{\mathrm{C}}
\newcommand{\cp}{\mathrm{cp}}
\newcommand{\avind}{\overline{\ind}}
\newcommand{\ind}{\mathrm{ind}}
\newcommand{\nul}{\mathrm{nul}}
\newcommand{\W}{W^{1,2}}

\begin{document}

\title[Closed geodesics on non-compact manifolds]{On the existence of infinitely many\\ closed geodesics on non-compact manifolds}
\author{Luca Asselle}
\address{Ruhr Universit\"at Bochum, Fakult\"at f\"ur Mathematik\newline\indent Geb\"aude NA 4/33, D-44801 Bochum, Germany}
\email{luca.asselle@ruhr-uni-bochum.de}
\author{Marco Mazzucchelli}
\address{CNRS, \'Ecole Normale Sup\'erieure de Lyon, UMPA\newline\indent  69364 Lyon Cedex 07, France}
\email{marco.mazzucchelli@ens-lyon.fr}
\date{February 11, 2016. \emph{Revised}: July 22, 2016}

\begin{abstract}
We prove that any complete (and possibly non-compact) Riemannian manifold $M$ possesses infinitely many closed geodesics provided its free loop space has unbounded Betti numbers in degrees larger than $\dim(M)$, and there are no close conjugate points at infinity. Our argument builds on an existence result due to Benci and Giannoni, and generalizes the celebrated theorem of Gromoll and Meyer for closed manifolds.
\end{abstract}

\maketitle

\section{Introduction}

The closed geodesics conjecture in Riemannian geometry and Hamiltonian dynamics states that every closed Riemannian manifold of dimension larger than one possesses infinitely many closed geodesics. Before the 1960s, it was not even clear whether there could exist simply connected closed manifolds possessing infinitely many closed geodesics for any choice of a Riemannian metric. The celebrated theorem of Gromoll and Meyer \cite[Theorem~4]{Gromoll:1969gh} motivated the formulation of the closed geodesics conjecture, by confirming its validity for the class of simply connected closed manifolds $M$ whose free loop space has unbounded Betti numbers in some coefficient field $\F$. A result from rational homotopy theory due to Vigu\'e-Poirrier and Sullivan \cite{Vigue-Poirrier:1976ug} implies that such condition on the loop space homology with $\F=\Q$ is equivalent to the fact that the cohomology algebra $\Hom^*(M;\Q)$ require at least two generators. Among the spaces for which the closed geodesics conjecture is still open there are the $n$-dimensional spheres for $n>2$ (the case of $n=2$ being known from a combination of results due to Bangert \cite{Bangert:1993wo}, Franks \cite{Franks:1992jt} and Hingston \cite{Hingston:1993ou}).

Non-compact Riemannian manifolds may not have closed geodesics at all. This is the case, for instance, for the flat Euclidean spaces. Consequently, any existence and multiplicity result must impose further conditions on the topology of the manifold and on its geometry at infinity. In \cite{Bangert:1980ho}, Bangert proved that every complete Riemannian surface that has finite area and is homeomorphic to a plane, a cylinder, or a Mobi\"us band, possesses infinitely many closed geodesics. In \cite{Thorbergsson:1978la}, Thorbergsson proved the existence of a closed geodesic in any complete and non-contractible Riemannian manifold with non-negative sectional curvature outside a compact set. In \cite{Benci:1991kk, Benci:1992lq}, Benci and Giannoni proved with different techniques a variation of Thorbergsson's result: any complete $d$-dimensional Riemannian manifold possesses a closed geodesic provided the limit superior of its sectional curvature at infinity is non-positive and the homology of its free loop space is non-trivial in some degree larger than $2d$. 

In this short paper, we build on Benci-Giannoni's and Gromoll-Meyer's arguments to study closed geodesics in a class of complete and not necessarily compact Riemannian manifolds. The condition that we will require on the geometry of our Riemannian manifold at infinity is the following one.

\begin{defn}\label{d:close_conj_pts}
A Riemannian manifold $(M,g)$ is said to be \textbf{without close conjugate points at infinity} when for all $\ell>0$ there exists a compact subset $K_\ell\subset M$ such that, for every geodesic $\gamma:[0,1]\to M\setminus K_\ell$ of length at most $\ell$, the point $\gamma(0)$ has no conjugate points along $\gamma$.
\end{defn}

We refer the reader to Section~\ref{ss:Morse_idx} for the background on the classical notion of conjugate points from Riemannian geometry. The condition in Definition~\ref{d:close_conj_pts} is satisfied for instance if $(M,g)$ has no conjugate points outside a compact set, and is milder than the geometric assumption in Benci-Giannoni's result: if the limit superior of the sectional curvature of $(M,g)$ at infinity is non-positive, then $(M,g)$ is without close conjugate points at infinity; this follows from Rauch comparison theorem, see Proposition~\ref{p:sectional_curvature}.

In \cite[Remark~1.4]{Benci:1992lq}, Benci and Giannoni speculated that their existence result above is still valid under a weaker condition on the free loop space homology: it is enough for such homology to be non-trivial in some degree larger than $d$. Our first result confirms their claim.

\begin{thm}\label{t:existence}
Let $(M,g)$ be a complete Riemannian manifold without close conjugate points at infinity, whose free loop space has nontrivial homology in some degree larger than $\dim(M)$ for some coefficient field. Then $(M,g)$ possesses at least a closed geodesic.
\end{thm}

Our second result is an extension of Gromoll and Meyer's theorem to a class of possibly non-compact Riemannian manifolds, that is, a confirmation of the closed geodesics conjecture for this class.

\begin{thm}\label{t:multiplicity}
Let $(M,g)$ be a complete, connected, Riemannian manifold without close conjugate points at infinity, whose free loop space has unbounded Betti numbers in degrees larger than $\dim(M)$ for some coefficient field. Then $(M,g)$ possesses infinitely many closed geodesics.
\end{thm}

The proofs of Theorems~\ref{t:existence} and~\ref{t:multiplicity} are purely based on techniques from critical point theory, and do not make use of the reversibility $|v|_g=|-v|_{g}$ of the Riemannian norm, unlike other results on closed geodesics (for instance Lusternik-Schnirelmann's theorem on the existence of three closed geodesics without self-intersections on Riemannian 2-spheres, see \cite{Grayson:1989gd, Taimanov:1992fe, Hass:1994pd}). It is well known that, if the reversibility is not exploited, there is no essential difference between Riemannian and Finsler manifolds from the point of view of critical point theory. The only extra difficulty arising in the Finsler case is merely technical: the energy function of closed geodesics is not $C^\infty$ as in the Riemannian case, but is only $C^{1,1}$ (see, e.g., \cite{Mercuri:1977fj, Rademacher:1992sw, Bangert:2010ak}). Such lack of regularity can be circumvented by passing to suitable finite dimensional approximations of the free loop space, as explained for instance in \cite[Section~3]{Bangert:2010ak}. Therefore, with minor cosmetic modifications in the proofs, Theorems~\ref{t:existence} and~\ref{t:multiplicity} also hold if we replace the Riemannian metric $g$ with a Finsler metric $F$ whose geodesic flow is complete (that is, defined for all times).

\subsection*{Acknowledgments} 
M.M. is grateful to Jean-Claude Sikorav for a fruitful conversation about the statement of Theorem~\ref{t:multiplicity}, as well as for many other discussions over the past few years. L.A. is partially supported by the DFG grant AB 360/2-1 ``Periodic orbits of conservative systems below the Ma\~n\'e critical energy value''.
M.M. is partially supported by the ANR projects WKBHJ (ANR-12-BS01-0020) and COSPIN (ANR-13-JS01-0008-01).

\section{Preliminaries}

\subsection{The penalized energy}

Let $(M,g)$ be a complete Riemannian manifold. Its 1-periodic closed geodesics (possibly stationary curves) are precisely the critical points of the energy function 
\begin{align*}
E:\Lambda M:=\W(\R/\Z,M)\to[0,\infty),\qquad
E(\gamma)=\int_0^1 g_{\gamma(t)}(\dot\gamma(t),\dot\gamma(t))\,\diff t.
\end{align*}
Throughout this paper, by ``geodesic'' we will always mean ``geodesic parametrized with constant speed''. We equip the free loop space $\Lambda M$ with the complete Hilbert-Riemannian metric induced by $g$. The function $E$ is smooth, but may not satisfy the Palais-Smale condition with respect to this metric if $M$ is not compact, as there may be Palais-Smale sequences of loops that escape towards the ends of the manifold $M$. Following Benci and Giannoni \cite{Benci:1991kk, Benci:1992lq}, we provide the needed compactness by means of the following  penalization trick. Let $\{\rho_\alpha\,|\,\alpha\in\N\}$ be a partition of unity such that the support of each $\rho_\alpha$ is compact. We define the smooth function
$f_\alpha:M\to[0,\infty)$ by 
\begin{align*}
f_\alpha(x):=\sum_{\beta>\alpha} \beta\cdot\rho_\beta(x).
\end{align*}
Notice that $f_\alpha$ is a non-negative proper function whose support is the union of the compact sets $\mathrm{supp}(\rho_\beta)$, for $\beta>\alpha$. In particular, since the family 
$\{\mathrm{supp}(\rho_\alpha)\,|\,\alpha\in\N\}$ is a locally finite cover of $M$, for every compact subset $K\subset M$ there exists $\alpha_0=\alpha_0(K)\in\N$ such that, for all integers $\alpha\geq\alpha_0$, the support of $f_\alpha$ is disjoint from $K$. For each $\alpha$, we introduce the \textbf{penalized energy}
\begin{align*}
E_\alpha:\Lambda M\to[0,\infty),\qquad
E_\alpha(\gamma)=E(\gamma) + f_\alpha(\gamma(0)),
\end{align*}
which is clearly smooth. For every sequence $\{\gamma_n\in\Lambda M\,|\,n\in\N\}$ such that $\gamma_n(0)$ tends to the ends of the manifold $M$ as $n\to\infty$, we have $f_\alpha(\gamma_n(0))\to\infty$. This is enough to provide the compactness that was lacking for the sublevel sets of $E$, and indeed the functions $E_\alpha$ satisfy the Palais-Smale condition: every sequence $\{\gamma_n\in\Lambda M\ |\ n\in\N\}$ such that $E_\alpha(\gamma_n)$ is uniformly bounded and $\|\diff E(\gamma_n)\|$ is infinitesimal admits a converging subsequence. We refer the reader to \cite[Lemma~4.5]{Benci:1992lq} for a proof of this fact. The critical points of $E_\alpha$ are those curves $\gamma$ that restrict to geodesics on the open interval $(0,1)\subset\R/\Z$ and satisfy $\dot\gamma(0^-)-\dot\gamma(0^+)=\grad f_\alpha(\gamma(0))$, where ``$\grad$'' denotes the $g$-gradient. In particular, every such $\gamma$ is a closed geodesic if and only if $\gamma(0)$ is a critical point of $f_\alpha$, for instance if $\gamma(0)\not\in\mathrm{supp}(f_\alpha)$.

\subsection{Morse indices}\label{ss:Morse_idx}

Let $\gamma$ be a critical point of $E$, that is, a 1-periodic closed geodesic. The Hessian of $E$, seen as a bounded self-adjoint operator on the Hilbert space $\Tan_\gamma(\Lambda M)$, is a compact perturbation of the identity (see, e.g., \cite[Theorem~2.4.2]{Klingenberg:1978so}), and in particular it has pure point spectrum with finitely many negative eigenvalues. The \textbf{Morse index} $\ind(E,\gamma)$ is defined as the number of negative eigenvalues of such Hessian counted with multiplicity, and the \textbf{nullity} $\nul(E,\gamma)$ is defined as the dimension of its kernel\footnote{Since $\nul(E,\gamma)$ is always larger than or equal to $1$, in the literature  the nullity of a closed geodesic is sometimes defined as $\nul(E,\gamma)-1$.}. Everything discussed so far works analogously if we replace $E$ with the penalized energy function $E_\alpha$. The last index that we will need to consider is the \textbf{local homology} 
$\Loc_*(E,S^1\cdot\gamma)$, which is defined as the relative homology group (with coefficients in some field)
\begin{align*}
\Loc_*(E,S^1\cdot\gamma):=\Hom_*\big (\{E<c\}\cup S^1\cdot\gamma,\{E<c\}\big ),
\end{align*}
where $c:=E(\gamma)$, and $S^1\cdot\gamma:=\{\gamma(t+\cdot)\in\Lambda M\,|\,t\in\R/\Z\}$ is the critical circle of $E$ containing $\gamma$. For the background on local homology groups, we refer the reader to \cite{Gromoll:1969jy} or to \cite[Chapter~1]{Chang:1993ng}. Here, we just recall that the local homology $\Loc_d(E,S^1\cdot\gamma)$ is always trivial in degrees $d<\ind(\gamma)$ and $d>\ind(\gamma)+\nul(\gamma)$.

The Morse indices and nullities of $E$ admit a symplectic interpretation as follows. Let $\phi_t$ be the geodesic flow of the Riemannian metric $g$ on the tangent bundle $\Tan M$. The nullity of a closed geodesic $\gamma\in\Lambda M$ is given by
\begin{align}\label{e:nullity}
\nul(E,\gamma)=\dim\ker\big(\diff\phi_1(\gamma(0),\dot\gamma(0))-\mathrm{Id}\big).
\end{align}
We denote by $V:=\ker(\diff\pi)\subset \Tan\Tan M$ the vertical subbundle, where  $\pi:\Tan M\to M$ is the base projection. We recall that two points $(x,v)\in\Tan M$ and $\phi_t(x,v)\in\Tan M$, for some $t>0$, are said to be \textbf{conjugate} when the intersection $\diff\phi_t(x,v)V_{(x,v)}\cap V_{\phi_t(x,v)}$ is non-trivial; if we set $\zeta(s):=\phi_s(x,v)$ for $s\in\R$, we will say that $\zeta(0)$ and $\zeta(t)$ are conjugate along the geodesic arc $\zeta|_{[0,t]}$. 
We denote by $\cp_t(x,v)$ the total multiplicity of points $\zeta(s)$, for $s\in(0,t]$, that are conjugate to $\zeta(0)$ along $\zeta|_{[0,s]}$. Namely,
\begin{align*}
\cp_t(x,v)=\sum_{s\in(0,t]} \dim\big(\diff\phi_s(x,v)V_{(x,v)}\cap V_{\phi_s(x,v)}\big).
\end{align*}
Let $\Omega M=\{\zeta\in\Lambda M\,|\,\zeta(0)=x\}$ be the space of loops based at a fixed point $x\in M$. The critical points of the restricted energy $E|_{\Omega M}$ are precisely the curves $\gamma\in\Omega M$ that are smooth geodesics outside time 0. By the Morse index Theorem for geodesics \cite[Theorem~15.1]{Milnor:1963rf}, we have
\begin{align}
\label{e:based_Morse_idx}
\ind\big (E|_{\Omega M},\gamma\big ) & = 
\lim_{t\to1^-}
\cp_t(\gamma(0),\dot\gamma(0^+)).
\end{align}

\begin{lem}\label{l:Morse_index_bound}
For every critical point $\gamma$ of $E_\alpha$ such that $\cp_1(\gamma(0),\dot\gamma(0^+))=0$, we have $\ind(E_\alpha,\gamma)+\nul(E_\alpha,\gamma)\leq \dim(M)$.
\end{lem}

\begin{proof}
We denote by $H$ the second derivative of the penalized action $E_\alpha$ at the critical point $\gamma$, which is a bounded symmetric bilinear form on the Hilbert space $\Tan_\gamma(\Lambda M)$. A straightforward computation shows that
$$H(\xi,\eta)
=
\int_0^1
\Big[ g_{\gamma(t)}\big (\nabla\xi(t),\nabla\eta(t)\big ) - g_{\gamma(t)}\big (R(\xi(t),\dot\gamma(t))\dot\gamma(t),\eta(t)\big ) \Big]\diff t + h(\xi(0),\eta(0)),$$
where $\nabla$ denotes the $g$-covariant derivative, $R$ the $g$-Riemann tensor, and $h$ the $g$-Riemannian Hessian of $f_\alpha$ at $\gamma(0)$, that is, the symmetric bilinear form on $\Tan_{\gamma(0)}M$ uniquely defined by
\begin{align*}
h(v,v)=\frac{\diff^2}{\diff s^2}\bigg|_{s=0} f_\alpha(\exp_{\gamma(0)}(s v)). 
\end{align*}
Since $\gamma$ is a smooth geodesic outside time 0, it is a critical point of the restricted energy $E|_{\Omega M}$. Consider the tangent space $\Tan_\gamma(\Omega M)$, which is the subspace of $\Tan_\gamma(\Lambda M)$ given by the $\W$-vector fields $\xi$ along $\gamma$ such that $\xi(0)=\xi(1)=0$. We denote by $\Tan_\gamma(\Omega M)^\bot$ the orthogonal to $\Tan_\gamma(\Omega M)$ with respect to the bilinear form $H$, that is,
\begin{align*}
\Tan_\gamma(\Omega M)^\bot
=
\Big\{
\xi\in\Tan_\gamma(\Lambda M)\ \Big|\ H(\xi,\eta)=0,\quad \forall \eta\in\Tan_\gamma(\Omega M)
\Big\}.
\end{align*}
A standard bootstrapping argument, together with an integration by parts, implies that $\Tan_\gamma(\Omega M)^\bot$ is the space of continuous 1-periodic vector fields $\xi$ along $\gamma$ that are smooth Jacobi fields outside time 0. This means that
\begin{align*}
\diff\phi_t(\xi(0),\nabla\xi(0^+)) & =(\xi(t),\nabla\xi(t)),\qquad\forall t\in(0,1),\\
\xi(0) & = \xi(1).
\end{align*}
If $\xi_1$ and $\xi_2$ are two such Jacobi fields such that $\xi_1(0)=\xi_2(0)$, their difference $\eta=\xi_1-\xi_2$ is a Jacobi field such that $\eta(0)=\eta(1)=0$. The assumption $\cp_1(\gamma(0),\dot\gamma(0^+))=0$ forces $\eta$ to vanish identically. This shows that
\begin{align}
\label{e:dimension_orthogonal}
\dim\big (\Tan_\gamma(\Omega M)^\bot\big )\leq \dim(M).
\end{align}
By definition, the Morse index $\ind(E_\alpha,\gamma)$ and the nullity $\nul(E_\alpha,\gamma)$ are respectively equal to the negative inertia index $\ind(H)$ and to the nullity $\nul(H)$ of the symmetric bilinear form $H$. Similarly, $\ind(E_\alpha|_{\Omega M},\gamma)$ and $\nul(E_\alpha|_{\Omega M},\gamma)$ are respectively equal to $\ind(H|_{\Tan_\gamma(\Omega M)\times\Tan_\gamma(\Omega M)})$ and to $\nul(H|_{\Tan_\gamma(\Omega M)\times\Tan_\gamma(\Omega M)})$. 
By a standard linear algebra formula for the inertia indices of restricted quadratic forms (see, e.g., \cite[Propositions~A.2--A.3]{Mazzucchelli:2015zc}), we have 
\begin{align*}
\ind(H) + \nul(H)  =\ & \ind\big (H|_{\Tan_\gamma(\Omega M)\times\Tan_\gamma(\Omega M)}\big ) + \ind\big (H|_{\Tan_\gamma(\Omega M)^\bot\times\Tan_\gamma(\Omega M)^\bot}\big )\\
& + \nul\big (H|_{\Tan_\gamma(\Omega M)^\bot\times\Tan_\gamma(\Omega M)^\bot}\big ). 
\end{align*}
Since $\cp_1(\gamma(0),\dot\gamma(0^+))=0$, equation~\eqref{e:based_Morse_idx} implies
\begin{align*}
\ind\big (H|_{\Tan_\gamma(\Omega M)\times\Tan_\gamma(\Omega M)}\big )
&=\cp_1(\gamma(0),\dot\gamma(0^+))=0.
\end{align*}
Finally, by equation~\eqref{e:dimension_orthogonal}, we have
\begin{align*}
\ind\big (H|_{\Tan_\gamma(\Omega M)^\bot\times\Tan_\gamma(\Omega M)^\bot}\big )
+ \nul\big (H|_{\Tan_\gamma(\Omega M)^\bot\times\Tan_\gamma(\Omega M)^\bot}\big )
& \leq
\dim\big (\Tan_\gamma(\Omega M)^\bot\big )\\
& \leq
\dim(M).
\qedhere
\end{align*}
\end{proof}

We recall that the sectional curvature of a Riemannian manifold $(M,g)$ is the function $\kappa$ defined on the tangent two planes in $\Tan M$ as follows: for every linearly independent $v,w\in\Tan_x M$, we have
\begin{align*}
 \kappa_x\big (\mathrm{span}\{v,w\}\big ) = \frac{ g_x(R(v,w)w,v) }{ g_x(v,v)\,g_x(w,w) - g_x(v,w)^2},
\end{align*}
where $R$ is the Riemann tensor of $(M,g)$. We close this subsection by showing that Benci-Giannoni's condition on the sectional curvature at infinity implies our condition in Definition~\ref{d:close_conj_pts}.

\begin{prop}\label{p:sectional_curvature}
Let $(M,g)$ be a complete Riemannian manifold such that, for all non-compact sequences $\{x_n\, |\, n\in\N\}\subset M$, we have
\begin{align*}
\limsup_{n\to\infty} \kappa_{x_n} \leq 0.
\end{align*}
Then $(M,g)$ is without close conjugate points at infinity.
\end{prop}

\begin{proof}
Consider a positive length $\ell>0$, and let $K_\ell\subset M$ be a large enough compact subset such that $\kappa_x < (\pi/\ell)^2$ for all $x\in M\setminus K_\ell$. As a direct consequence of Rauch comparison theorem (see \cite[Proposition~2.4]{Carmo:1992jk}), any geodesic $\gamma:[0,1]\to M\setminus K_\ell$ of length smaller than or equal to $\ell$ does not contain a pair of conjugate points.
\end{proof}

\subsection{Iterated closed geodesics}

The $m$-th iterate of $\gamma\in\Lambda M$ is the curve $\gamma^m\in\Lambda M$ given by $\gamma^m(t)=\gamma(mt)$. If $\gamma$ is a closed geodesic, and thus a critical point of $E$, $\gamma^m$ is a critical point of $E$ as well. A classical result of Bott \cite{Bott:1956sp} (see also \cite{Long:2002ed, Mazzucchelli:2015zc} for more modern accounts) implies that the Morse indices $\ind(E,\gamma^m)$ satisfy the iteration inequalities
\begin{align}\label{e:iteration_inequalities}
m\,\avind(\gamma)-\dim(M) \leq \ind(E,\gamma^m) \leq m\,\avind(\gamma)+\dim(M) - \nul(E,\gamma^m),
\end{align}
where the non-negative real number $\avind(\gamma)$ is the \textbf{average index} of $\gamma$, defined by 
\begin{align*}
\avind(\gamma) = \lim_{m\to\infty} \frac{\ind(E,\gamma^m)}{m}.
\end{align*}
As for the nullity, equation~\eqref{e:nullity} together with an argument from linear algebra (see, e.g., \cite[Proposition~A.1]{Mazzucchelli:2015zc}) implies that there exists a partition $\N=\N_1\cup...\cup\N_k$ such that, for each $i=1,...,k$, every integer $m\in\N_i$ is a multiple of $m_i:=\min\N_i$, and 
\begin{align*}
\nul(E,\gamma^m)=\nul(E,\gamma^{m_i}),\qquad\forall m\in\N_i.
\end{align*}
A result of Gromoll and Meyer \cite[Theorem~3]{Gromoll:1969gh} implies, for all $m\in\N_i$, the local homology $\Loc_{*}(E,\gamma^m)$ is isomorphic to $\Loc_{*}(E,\gamma^{m_i})$ up to a shift in the degree, and in particular
\begin{align}\label{e:Gromoll_Meyer}
\dim\big(\Loc_{*}(E,\gamma^m)\big)=\dim\big(\Loc_{*}(E,\gamma^{m_i})\big),\qquad\forall m\in\N_i.
\end{align}

\section{Proofs of the Theorems}

\subsection{Proof of Theorem~\ref{t:existence}}\label{ss:existence}
Our argument for the proof of Theorem~\ref{t:existence} follows closely the one in \cite{Benci:1992lq}, but employs the sharper Morse index bound given by Lemma~\ref{l:Morse_index_bound}. For each $\alpha\in\N$, we define the minimax function
\begin{equation}\label{e:minimax}
\begin{gathered}
 c_\alpha:\Hom_*(\Lambda M)\setminus\{0\}\to[0,\infty),
\\
c_\alpha(h):=\inf\Big\{ \ell^2\geq 0\ \Big|\ h\in \iota^\ell_*\Hom_*(\{E_\alpha<\ell^2\})\Big\},
\end{gathered}
\end{equation}
where $\iota^\ell:\{E_\alpha<\ell^2\}\hookrightarrow\Lambda M$ denotes the sublevel set inclusion. Since the penalized energy $E_\alpha$ satisfies the Palais-Smale condition, the Minimax Lemma \cite[page~79]{Hofer:1994bq} guarantees that $c_\alpha(h)$ is a critical value of $E_\alpha$, that is, the squared length of a (not necessarily smoothly periodic) geodesic that closes up in time 1. Since $E_\alpha\geq E_{\alpha+1}$ pointwise, we have
\begin{align}\label{e:monotonicity_minimax}
 c_\alpha(h)\geq c_{\alpha+1}(h),\qquad\forall\alpha\in\N,\ h\in\Hom_*(\Lambda M)\setminus\{0\}.
\end{align}
By assumption, there exist a degree $d>\dim(M)$ and a non-zero homology class $h\in\Hom_d(\Lambda M)$. We fix $\ell>c_1(h)^{1/2}$, and an integer $\alpha$ large enough so that
\begin{align}\label{e:shell}
 \dist(x,y)>\ell,\qquad\forall x\in \mathrm{supp}(f_{\alpha}),\ y\in K_\ell,
\end{align}
where $K_\ell\subset M$ is the compact subset given by Definition~\ref{d:close_conj_pts}. Notice that this is possible since the Riemannian manifold $(M,g)$ is complete. Equation~\eqref{e:monotonicity_minimax} implies that $c_\alpha(h)\leq c_1(h)\leq \ell^2$, and therefore the homology class $h$ belongs to the image of the inclusion-induced homomorphism $\iota^\ell_*:\Hom_d(\{E_\alpha<\ell^2\})\to\Hom_d(\Lambda M)$. In particular,
\begin{align}\label{e:nontrivial_homology}
\Hom_d\big (\{E_\alpha<\ell^2\}\big )\neq0.
\end{align}
By equation~\eqref{e:shell}, for each critical point $\gamma$ of $E_\alpha$ with $E_\alpha(\gamma)<\ell^2$, exactly one of the following cases is verified:
\begin{itemize}
\item[(i)] $\gamma(0)\in M\setminus \mathrm{supp}(f_\alpha)$, and therefore $\gamma$ is a critical point of $E$,
\item[(ii)] $\gamma(0)\in \mathrm{supp}(f_\alpha)$, and therefore the curve $\gamma$ is entirely contained in $M\setminus K_\ell$.
\end{itemize}
Assume by contradiction that every such critical point $\gamma$ satisfies (ii). Since $\gamma$ is contained in $M\setminus K_\ell$ and its length is less than $\ell$, we have $\cp_1(\gamma(0),\dot\gamma(0^+))=0$. Lemma~\ref{l:Morse_index_bound} implies that $\ind(E_\alpha,\gamma)+\nul(E_\alpha,\gamma)\leq \dim(M)$. In particular, since $d$ is larger than $\dim(M)$, no such critical point of $E_\alpha$ ``contributes'' to the homology group $\Hom_d(\{E_\alpha<\ell^2\})$. Therefore $\Hom_d(\{E_\alpha<\ell^2\})$ must be trivial, in contradiction with~\eqref{e:nontrivial_homology}.

\subsection{Proof of Theorem~\ref{t:multiplicity}}
We prove Theorem~\ref{t:multiplicity} by contradiction, assuming that the only non-iterated closed geodesics of $(M,g)$ are $\gamma_1,...,\gamma_r$. Namely, each critical point of the energy $E$ with positive critical value belongs to a critical circle of the form $S^1\cdot\gamma_i^m$ for some $i\in\{1,...,r\}$ and $m\in\N$. By Gromoll and Meyer's equation~\eqref{e:Gromoll_Meyer}, there exists a constant $u\in\N$ such that the local homology of each critical circle $S^1\cdot\gamma_i^m$ has dimension at most $u$, i.e.
\begin{align}\label{e:local_homology_bound}
\dim\big( \Loc_*(E,S^1\cdot\gamma_i^m) \big) \leq u,
\qquad
\forall i\in\{1,...,r\},\ m\in\N.
\end{align}
We denote by $a>0$ the minimum among the positive average indices of the closed geodesics $\gamma_1,...,\gamma_r$, i.e.
\begin{align*}
a:=\min\Big( \left\{\avind(\gamma_1),...,\avind(\gamma_r)\right\}\cap (0,\infty) \Big). 
\end{align*}

\begin{claim}\label{c:1}
For every degree $d>\dim(M)$, there are at most $\lfloor1+2\dim(M)/a\rfloor r$ critical circles of $E$ having non-trivial local homology in degree $d$.
\end{claim}

\begin{proof}
We recall that the local homology of a critical circle $S^1\cdot\gamma$ is trivial in all degrees outside the interval $[\ind(E,\gamma),\ind(E,\gamma)+\nul(E,\gamma)]$. By the rightmost iteration inequality in~\eqref{e:iteration_inequalities}, if $\avind(\gamma)=0$ we have $\ind(\gamma)+\nul(\gamma)\leq\dim(M)$, and in particular the local homology $\Loc_d(E,S^1\cdot\gamma)$ is trivial. Therefore it is enough to prove that, for each  $\gamma_i$ such that $\avind(\gamma_i)>0$, there are at most $\lfloor1+2\dim(M)/\avind(\gamma_i)\rfloor$ integers $m\in\N$ such that the interval $[\ind(E,\gamma_i^m),\ind(E,\gamma_i^m)+\nul(E,\gamma_i^m)]$ contains $d$. This readily follows from the iteration inequalities~\eqref{e:iteration_inequalities}.
\end{proof}

\begin{claim}\label{c:2}
For all $\ell>0$ there exists $\overline\alpha=\overline\alpha(\ell)\in\N$ such that, for all integers $\alpha>\overline\alpha$ and $d>\dim(M)$, we have
\begin{align*}
\dim\big(\Hom_d(\{E_\alpha<\ell^2\})\big) \leq \lfloor1+2\dim(M)/a\rfloor ru.
\end{align*}
Here, $\Hom_*$ denotes the singular homology functor with coefficients in any field.
\end{claim}

\begin{proof}
We proceed as in the last paragraph of Section~\ref{ss:existence}. We choose the integer $\overline\alpha$ large enough so that $\dist(x,y)>\ell$ for all $x\in \mathrm{supp}(f_{\overline\alpha})$ and $y\in K_\ell$. We fix $\alpha>\overline\alpha$. Since $\mathrm{supp}(f_\alpha)\subset\mathrm{supp}(f_{\overline\alpha})$, for each critical point $\gamma$ of $E_\alpha$ with $E_\alpha(\gamma)<\ell^2$, either point (i) or point (ii) of Section~\ref{ss:existence} is verified. If $\gamma$ satisfies point~(ii) we have $\cp_1(\gamma(0),\dot\gamma(0^+))=0$, and therefore Lemma~\ref{l:Morse_index_bound} implies that $\ind(E_\alpha,\gamma)+\nul(E_\alpha,\gamma)\leq \dim(M)$; in particular, since $d>\dim(M)$, such critical point $\gamma$ does not ``contribute'' to the homology group $\Hom_d(\{E_\alpha<\ell^2\})$. Therefore, we have the Morse inequality 
\begin{align*}
\dim\big( \Hom_d(\{E_\alpha<\ell^2\}) \big)
\leq
\sum_{i=1}^r \sum_{m\in\N}
\dim\big( \Loc_d(E,S^1\cdot\gamma_i^m) \big).
\end{align*}
By Claim~1, for each $i$ there are at most $\lfloor1+2\dim(M)/a\rfloor r$ values of $m$ such that the term of the above sum is non-zero. This, together with~\eqref{e:local_homology_bound}, implies Claim~\ref{c:2}.
\end{proof}

From now on, let us work in a coefficient field $\F$ for the singular homology $\Hom_*$ such that 
\begin{align*}
\sup_{d>\dim(M)} \dim\big(\Hom_d(\Lambda M)\big)=\infty.
\end{align*}
In particular, there exists a degree $d>\dim(M)$ such that 
\begin{align*}
\dim\big(\Hom_d(\Lambda M)\big)>\lfloor1+2\dim(M)/a\rfloor ru.
\end{align*}
We choose arbitrary homology classes 
\[h_1,...,h_{\lfloor1+2\dim(M)/a\rfloor ru+1}\in\Hom_d(\Lambda M)\] 
that are linearly independent. We consider the minimax functions defined in~\eqref{e:minimax}, and fix a real number $\ell$ such that
\begin{align*}
\ell^2 > \max\big\{c_1(h_1),...,c_1(h_{\lfloor1+2\dim(M)/a\rfloor ru+1})\big\}.
\end{align*}
By~\eqref{e:monotonicity_minimax}, we have
\begin{align*}
\ell^2 > \max\big\{c_\alpha(h_1),...,c_\alpha(h_{\lfloor1+2\dim(M)/a\rfloor ru+1})\big\},\qquad \forall\alpha\in\N.
\end{align*}
In particular, for all $\alpha\in\N$, each homology class $h_i$ belongs to the image of the inclusion-induced homomorphism $\iota^\ell_*:\Hom_d(\{E_\alpha<\ell^2\})\to\Hom_d(\Lambda M)$, and therefore
\begin{align*}
\dim \big(\Hom_d(\{E_\alpha<\ell^2\}) \big)\geq \lfloor1+2\dim(M)/a\rfloor ru + 1,\qquad\forall\alpha\in\N.
\end{align*}
This contradicts Claim~\ref{c:2}, and therefore completes the proof of Theorem~\ref{t:multiplicity}.

\bibliography{_biblio}
\bibliographystyle{amsalpha}

\end{document}